\documentclass[11pt,a4paper]{article}
\usepackage{palatino}
\usepackage{amsmath}
\usepackage{amssymb}
\usepackage{amsthm}
\usepackage{amsfonts}
\usepackage{enumerate}
\usepackage{marvosym}

\usepackage{graphicx}
\usepackage[all]{xy}
\usepackage{xypic}

\usepackage{anysize}
\marginsize{3cm}{2.5cm}{2.5cm}{2.5cm}
\usepackage{color}
\usepackage{graphicx}
\usepackage[all]{xy}
\usepackage{xypic}

\newtheorem{Theorem}{Theorem}[section]
\newtheorem{Corollary}[Theorem]{Corollary}
\newtheorem{Lemma}[Theorem]{Lemma}
\newtheorem{Proposition}[Theorem]{Proposition}
\newtheorem{Definition}[Theorem]{Definition}

\begin{document}

\title{Dynamical systems associated to the $\beta$-core in the repeated prisoner's dilemma}
\author{S\l awomir Plaskacz\\
\small Faculty of Mathematics and Computer Science\\
\small N. Copernicus University in Toru\'n\\
\small e-mail: plaskacz@mat.umk.pl
\and Joanna Zwierzchowska\\
\small Faculty of Mathematics and Computer Science\\
\small N. Copernicus University in Toru\'n\\
\small e-mail: joanna.zwierzchowska@mat.umk.pl}
\date{}
 \maketitle

 \vspace{3mm}


 \vspace{3mm}

{\bf Abstract : }
We consider the repeated prisoner's dilemma (PD). We assume that players make their choices knowing only average payoffs from the previous stages. A player's strategy is a function from the convex hull $\mathfrak{S}$ of the set of payoffs into the set $\{C,\,D\}$ ($C$ means cooperation, $D$ -- defection). S. Smale in \cite{smale} presented an idea of good strategies in the repeated PD. If both players play good strategies then the average payoffs tends to the payoff corresponding to the profile $(C,C)$ in PD. We adopt the Smale idea to define semi-cooperative strategies - players do not take as a referencing point the payoff corresponding to the profile $(C,C)$, but they can take an arbitrary payoff belonging to the $\beta$-core of PD. We show that if both players choose the same point in the $\beta$-core then the strategy profile is an equilibrium in the repeated game. If the players choose different points in the $\beta$-core then the sequence of the average payoffs tends to a point in $\mathfrak{S}$. The obtained limit can be treated as a payoff in a new game. In this game the set of players' actions is the set of points in $S$ that corresponds to the $\beta$-core payoffs.

\section{Introduction}
Robert Aumann in papers \cite{aumann20}, \cite{aumann21}, \cite{aumanncore}, \cite{aumannsurvey} showed that if a payoff $p$ of a normal form game corresponds to a strategy profile belonging to the $\beta$-core then there exists a strong equilibrium in the repeated game providing the payoff $p$. The construction of the strong equilibrium profile in the repeated game is rather complex and bases on the assumption that all players know the full history of the game. Our aim is to consider also the situation when players choose strategies   in the repeated game corresponding to different points in the $\beta$-core.  Then, in general,  the course of the game seems  heavy to forecast.

We consider two players Prisoner's Dilemma (PD) with payoffs given by
\begin{equation}
\begin{array}{c|c|c}
 & \mbox{C} & \mbox{D} \\
\hline
\mbox{C} & (2,2) & (0,3)\\
\hline
\mbox{D} & (3,0) & (1,1)
\end{array}\label{staticPD}
\end{equation}
where $C$ means \textit{to cooperate} and $D$ -- \textit{to defect}. The set of $\beta$-core payoffs for PD is presented in Figure \ref{beta} (bold segments with ends $(1,\,2.5),\,(2,\,2),\,(2.5,\,1)$).
\begin{figure}[!ht]
  \centering
    \includegraphics[width=0.5\textwidth]{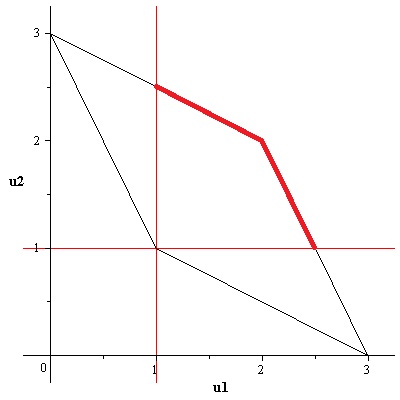}
  \caption{The $\beta$-core of the prisoner's dilemma}
\label{beta}
\end{figure}

The $\beta$-core consists of payoffs which are Pareto optimal and individually rational. We assume that in the repeated game players know only both players average payoffs from the previous stages. So, a player's strategy is a function from the convex hull $S$ of the set of vector payoffs into the set of his actions. The vector payoff function $u$ given by (\ref{staticPD}), the strategy profile $s:S\to\{C,\,D\}^2$ and an initial point $\bar{x}_1$ determine a sequence of average payoffs $\bar{x}_t$ by
\begin{equation}\label{sequence}
\bar{x}_{t+1}=\frac{t\bar{x}_t+u(s(\bar{x}_t))}{t+1}
\end{equation}
The strategies of player 1 and player 2  corresponding to a point $v=(v_1,\,v_2)$ in the $\beta$-core are presented in Figure \ref{semi}.
\begin{figure}
  \centering
  \includegraphics[width=0.8\textwidth]{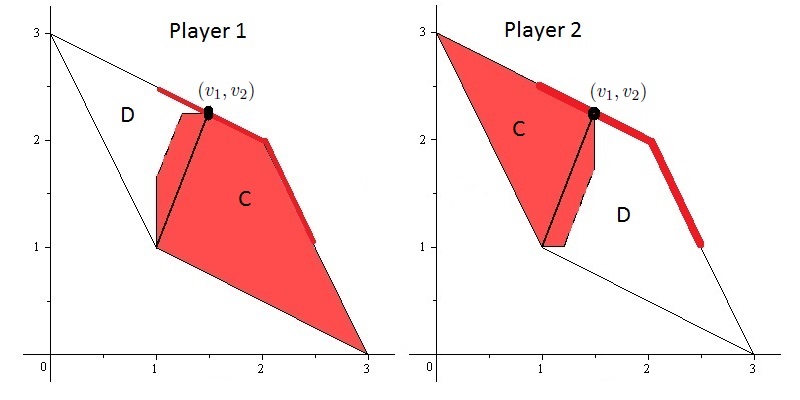}
 \caption{The semi-cooperative strategies}
\end{figure}\label{semi}
The strategies are called semi-cooperative strategies and are determined by the point $v$ and a positive constant $\varepsilon$. We show that for an arbitrary initial point $\bar{x}_1$, the sequence of average payoffs $\bar{x}_t$ is convergent to the point $v$ when the strategy profile $s=s^v$ consists of the semi-cooperative strategies corresponding to the point $v$. The profile $s^v$ is a Nash equilibrium. The case $v=(2,\,2)$ was considered by S. Smale in \cite{smale}. The idea of semi-cooperative strategies is motivated by Smale's idea of good strategies.

The main problem, that we consider in the paper, is to study the limit properties of the dynamical system given by (\ref{sequence}) in the case when players 1 and 2 choose different points $v$ in the $\beta$-core -- player 1 chooses point $a$ and player 2 chooses point $b$.
Our main result formulated in Theorem \ref{twzbiega} states that for an arbitrary initial point the sequence of average payoffs is convergent and the limit is given by
\[
\overline{x}_t\quad \xrightarrow[t\rightarrow \infty]{} \quad \left\{
\begin{array}{lcl}
a & \mbox{ if } & a_1=b_1  \\
(2,2) & \mbox{ if } & a_1 \leqslant 2 \leqslant b_1  \\
b&\mbox{ if } & a_1<b_1\leq 2 \\
a &\mbox{ if } & 2 \leq  a_1<b_1 \\
\approx (1,1) &\mbox{ if } & b_1<a_1
\end{array}\right.
\]
The Smale approach to the repeated PD has been recently applied in \cite{akin}. E. Akin showed that if player's 1 strategy $s_1$ is simple, i.e.
\[
s_1(x)=\left\{\begin{array}{ll}
D&\mbox{if $L(x)>0$}\\
C&\mbox{if $L(x)<0$}
\end{array}\right.
\]
where $L(x_1,x_2)=ax_1+bx_2+c$ is an affine map such that $L(1,1),\,L(3,0)\leq 0\leq L(2,2),\,L(0,3)$, then every sequence of average payoffs is attracted by the interval $\{x\in S:\;L(x)=0\}$, for arbitrary strategy of player 2. If both players adopt simple strategies then every sequence of average payoffs tends to a point being the intersection of separation lines. In \cite{akin}, the evolutionary dynamics is used to analyze competition among certain simple strategies. The game with the payoff given by Theorem \ref{twzbiega} has a continuous strategy set. In the next paper we intend to analyze its replicator dynamics using methods presented in \cite{rr}.

\section{Smale's good strategies in the repeated prisoner's dilemma}

In this section we provide a brief presentation of Smale's approach to the repeated prisoner's dilemma which was presented in  $\cite{smale}$.

Smale considers PD with payoffs  given by (\ref{staticPD}).
The players' actions are interpreted as follows: $C$ means \textit{to cooperate} and $D$ -- \textit{to defect}\footnote{Originally, in the paper $\cite{smale}$, Smale understands the game $(\ref{staticPD})$ in the meaning of the arms race. In his paper the action $C$ is marked with $E$ -- easy -- which means to disarm, and the action $D$ is marked with $T$ -- tough -- means to arm.}. The game is symmetric and the action $D$ dominates the action $C$ for each player ($3>2$ and $1>0$). The Nash equilibrium is the pair of action $(D,D)$ and the Nash payoff is $(1,1)$. The Nash payoff is not Pareto optimal. The Pareto frontier contains two segments: the first one is jointing $(0,3)$ to $(2,2)$, the second one -- $(2,2)$ to $(3,0)$. Smale distinguishes one Pareto optimal payoff $(2,2)$.

He constructs a strategy profile in the repeated PD that is a Nash equilibrium with the payoff equals to $(2,\,2)$. This kind of result can be treated as a special case of the Folk Theorem. What makes Smale's approach not typical  is the way of choosing actions in each repetition. At each stage, the players make their decision basing on the average vector payoffs from the previous repetitions. It means that the domain of strategies is no longer the set of histories but now it is the convex hull of the payoffs. Such strategies are called \textit{memory strategies}. Each player chooses his memory strategy before the iterated game is started. Players' strategies are fixed during the iteration.

Let the function $u:\{C,D\}^2\rightarrow \mathfrak{S} $ be given by $(\ref{staticPD})$ where  $\mathfrak{S}$ denotes the convex hull of all possible payoffs, i.e.  $\mathfrak{S}=conv\{(2,2),$ $(0,3),(3,0),(1,1)\}$.
A \textit{memory strategy} of  player $i$ is a map $s_i:\mathfrak{S} \rightarrow \{C,D\}$. A \textit{strategy profile} is the pair $s=(s_1,s_2):\mathfrak{S} \rightarrow \{C,D\}^2$. The strategy profile $s$ and an initial point $x_1\in \mathfrak{S}$  determines the course of the repeated game in the following way:
\begin{displaymath}
\overline{x}_1:=x_1,\;\;\; x_{t+1}:=u(s(\overline{x}_t)),\;\;\;\overline{x}_{t+1}:=\frac{x_1+\cdots+x_{t+1}}{t+1} \mbox{ for }t\geqslant 1.
\end{displaymath}
The sequence $(x_t)_{t\geqslant 1}$ is the sequence of payoffs  and the sequence $(\overline{x}_t)_{t\geqslant 1}$ is the sequence of average payoffs in the repeated game.

Fix $\varepsilon>0$. A \textit{good strategy} of player 1 is a map $s_1^*: \mathfrak{S} \rightarrow \{C,D\}$ given by
\[
s_1^*(a,b) =\left\{\begin{array}{ll}
C &\mbox{ if } b<a+\varepsilon \mbox{ and } a\geq 1 \mbox{ and } b \leq 2\\
D & \mbox{ elsewhere in } S
\end{array}\right.
\]
A \textit{good strategy} of player 2 is a map  $s_2^*: \mathfrak{S} \rightarrow \{C,D\}$ given by
\[
s_2^*(a,b) =\left\{\begin{array}{ll}
C &\mbox{ if } a<b+\varepsilon \mbox{ and } b\geq 1 \mbox{ and } a \leq 2\\
D & \mbox{ elsewhere in } S
\end{array}\right.
\]
The good strategies are illustrated on Figure $\ref{ryswypldyl}$.
\begin{figure}[!ht]
  \centering
    \includegraphics[width=0.8\textwidth]{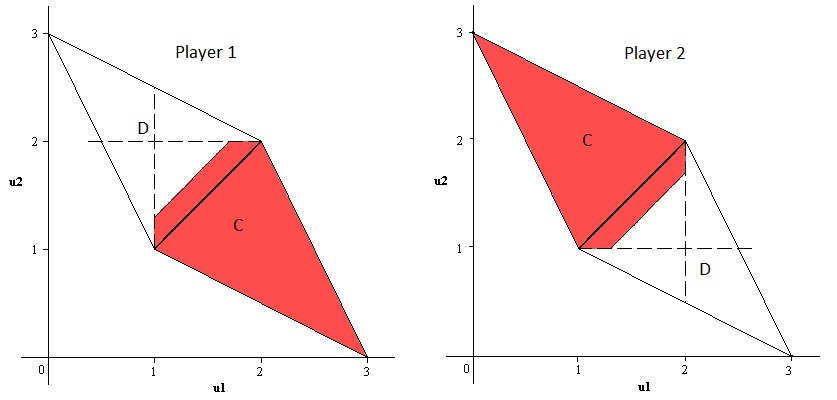}
  \caption{Good strategies}
\label{ryswypldyl}
\end{figure}

The main result presented in  section 1 of  $\cite{smale}$ is the following theorem.
\begin{Theorem}[Smale]\label{twSmale}
\begin{enumerate}
\item If player 1 plays a good strategy $s_1^*$ and player 2 plays an arbitrary strategy $s_2$ then the sequence of average payoffs
$\overline{x}_t =(\overline{x}^1_t,\overline{x}^2_t  )$ satisfies
\[\displaystyle
\liminf_{t\rightarrow \infty } \overline{x}^1_t \geqslant 1
\qquad \mbox{ and }\qquad
\limsup_{t\rightarrow \infty } \overline{x}^2_t \leqslant 2
\]
for every $x_1\in\mathfrak{S}$.
\item If both players play good strategies $s_1^*, s_2^*$, then
\[\displaystyle
\overline{x}_t \xrightarrow[t\rightarrow \infty]{} (2,2)
\]
for every $x_1\in\mathfrak{S}$.
\end{enumerate}
\end{Theorem}
If the payoff  in the repeated game is defined as the upper limit of average payoffs then the strategy profile $s^*=(s_1^*,\,s_2^*)$ is a Nash equilibrium in the set of memory strategies.

The Banach limit $Lim$ is a continuous linear functional defined on the space $l^\infty$ of bounded scalar sequences\footnote{The definition and properties of Banach limit can be found in $\cite{lax}$.} that is an extension of the functional which associates any convergent sequence with its limit.
If the payoff is defined as a Banach limit of the sequence of the average payoffs then the  Nash equilibrium $s^*$ has an additional interesting property. The construction of good strategies guarantees that the deviating player's payoff will not exceed the good strategy player's payoff by more than $\varepsilon$. We define the payoff in the repeated game by
\[
u^{\infty}_i(s,x_1)=Lim\, \overline{x}^i_t \qquad \mbox{ for } i=1,2.
\]

\begin{Proposition}
Suppose player 1 plays a good strategy $s_1^*$. If  $s=(s_1^*,s_2)$, where $s_2$ is an arbitrary memory strategy of player 2, then
\[\displaystyle
u^{\infty}_2(s,x_1)\leqslant u^{\infty}_1(s,x_1)+\varepsilon
\]
for every $x_1\in \mathcal{S}$.
\end{Proposition}
It means that if player 1 plays a good strategy, then his payoff is not smaller than his opponent's payoff minus $\varepsilon$. But the constant $\varepsilon$  is controlled by the player 1, so he can choose it as small as he wish. In this sense, we can say that good strategies not only are Nash equilibria in the set of memory strategies, but also they are \textit{safe Nash equilibria}.

\section{Some properties of the dynamical systems generated by memory strategies}\label{seclematy}

We consider a normal form game $G=\left(\mathcal{N},(A_i)_{i\in
\mathcal{N}},(u_i)_{i\in \mathcal{N}}\right)$, where
$\mathcal{N}=\{1,\ldots, N\}$ is the set of players, $A_i$ is a
finite set of actions of  player $i$ and $u_i:A=A_1\times
A_2\times \ldots \times A_N\rightarrow \mathbb{R}$
is the payoff function of  player $i$. \\
A memory strategy of player $i$ is a function
$s_i:\mathfrak{S}\rightarrow A_i$, where $\mathfrak{S}:=
conv\{u(a) \: : \: a\in A\}$ is the convex hull of the set of
vector payoffs $u=(u_1,\,u_2,\ldots,u_N)$. The strategy profile
$s=(s_1,\,s_2,\ldots,\,s_N)$ determines a map
$f^s:\mathfrak{S}\rightarrow \mathfrak{S}$ by
\begin{equation}\label{funkcjaf}
f^s(x)=(u\circ s)(x).
\end{equation}
and the dynamical system $\beta^s=(\beta^s_t)_{t\geq 1}$
\begin{equation}\label{procesdynam}
\beta_t^s(x)=\frac{tx+f^s(x)}{t+1}.
\end{equation}
We say that a sequence $(\overline{x}_t)_{t\geq t_0}\subset \mathfrak{S}$ is a trajectory of the dynamical system $\beta^s$ if
\begin{equation}\label{trajectory}
\overline{x}_{t+1}=\beta^s_t(\overline{x}_t)
\end{equation}
Observe that if $\overline{x}_{t_0}$ is the given average payoff after stage $t_0$ then the trajectory $(\overline{x}_t)_{t\geq t_0}\subset \mathfrak{S}$  of the dynamical system $\beta^s$ given by (\ref{trajectory}) is the sequence of average payoffs
\[
\overline{x}_{t+1}=\frac{t_0 \overline{x}_{t_0}+f^s
(\overline{x}_{t_0})+\ldots+f^s(\overline{x}_{t})}{t+1}
\]

Since
\[
\beta^s_t(\bar{x})-\bar{x}=\frac{f^s(\bar{x})-\bar{x}}{t+1}
\]
and the set $\mathfrak{S}$ is bounded, then
\begin{equation}\label{malykrok1}
\forall_{\varepsilon>0}\quad \exists_T\quad\forall_{t>T}\quad\forall_{\bar{x}\in\mathfrak{S}}
\;\;\;|\beta^s_t(\bar{x})-\bar{x}|<\varepsilon
\end{equation}
So, for every $\varepsilon>0$, there exists $T$ such that for an arbitrary trajectory $(\overline{x}_t)_{t\geq t_0}$  of the dynamical system $\beta^s$ it holds
\begin{equation}\label{malykrok}
\forall_{t>\max\{t_0,T\}}\;\;\; |\overline{x}_{t+1}-\overline{x}_t|<\varepsilon.
\end{equation}

The following  proposition is the deterministic version of the  Blackwell approachability result (see \cite{black}). An elementary proof is presented in \cite{pz}.

\begin{Proposition}\label{lemBlackwell}
Suppose that a set $W\subset \mathfrak{S}$ is closed and a trajectory $(\overline{x}_t)_{t\geq t_0}$  of the dynamical system $\beta^s$ satisfies
\begin{equation}\label{zallem}
\forall t\geq t_0,\,\exists y_t\in W,\;\; |\overline{x}_t-y_t|=\mbox{dist}(\overline{x}_t,W) \mbox{ and } \langle \overline{x}_t-y_t,\,f^s(\overline{x}_t)-y_t\rangle\leq 0.
\end{equation}
Then
\[
\lim_{t\rightarrow \infty} dist(\overline{x}_t,W)=0.
\]
\end{Proposition}

The point $y_t$ in (\ref{zallem}) is a proximal point in the set $W$ to the point $\overline{x}_t$. If the set $W$ is convex and closed and $f^s(\overline{x}_t)\in W$ for $t\geq t_0$ then (\ref{zallem}) holds true. As a corollary from Proposition \ref{lemBlackwell} we obtain that

\begin{Corollary}\label{lemuwypuk}
If the set $W\subset  \mathfrak{S}$ is closed and convex and  a trajectory $(\overline{x}_t)_{t\geq t_0}$  of the dynamical system $\beta^s$ satisfies
\begin{equation}\label{zalwypukl}
\exists_{\overline{t}\geq t_0}\quad \forall_{t>\overline{t}} \quad f^s(\overline{x}_t)\in W,
\end{equation}
then
\[
\forall_{\epsilon>0}\quad \exists_{t_{\epsilon}>\overline{t}} \quad \forall_{t>t_{\epsilon}} \quad \overline{x}_t\in W^{\epsilon}.
\]
\end{Corollary}
Taking $W=(-\infty,c]$ in Proposition \ref{lemBlackwell} we obtain the
following property of real sequences.
\begin{Corollary}\label{lBC}
Suppose that   $(a_n)_{n=1}^\infty$ is a bounded sequence in
 $ \mathbb{R}$ and $(\bar{a}_n)_{n=1}^\infty$ is the sequence of arithmetic means, i.e. $\bar{a}_n=\frac{1}{n}\sum_{k=1}^n a_k$.If we have

 \[
(\bar{a}_n>c \quad \Rightarrow \quad a_{n+1}\leq c)
 \]
for almost all $n$   and a fixed constant $c\in \mathbb{R}$, then
 \[
 \limsup_{n\rightarrow \infty }\bar{a}_n \leq c.
 \]
 \end{Corollary}

\begin{Definition}
Let  $s$ be a memory strategies profile. We say that a set $Z\subset\mathfrak{S} $ is:
\begin{enumerate}
\item invariant for the dynamical system $\beta^s$ iff
\[\exists_{t_Z\geq 1} \quad \forall_{x\in Z} \quad \forall_{t\geqslant t_Z} \quad \frac{tx+f^s(x)}{t+1}\in Z,
\]
\item an escape set for the dynamical system $\beta^s$ iff every trajectory $(\overline{x}_t)_{t\geq t_0}$ of the dynamical system $\beta^s$ satisfies
\[
\forall_{\tau\geq t_0}\quad \exists_{t>\tau} \quad \overline{x}_{t}\notin Z.
\]
\item an absorbing set for the dynamical system $\beta^s$ iff every trajectory $(\overline{x}_t)_{t\geq t_0}$ of the dynamical system $\beta^s$ satisfies
\[
\forall_{\tau\geq t_0}\quad \exists_{t>\tau} \quad \overline{x}_{t}\in Z.
\]
\end{enumerate}
\end{Definition}

 In the next section we study limit properties of some dynamical systems generated by memory strategies. To show that a  trajectory is convergent to a point  we will construct a family of absorbing and invariant neighbourhoods of the limit points. Invariance is usually easy to check. To show that a neighbourhood is absorbing we shall use the following lemmas.\\
 Hereafter to the end of the section we fix a memory strategies profile $s$ and consider the dynamical system $\beta^s$.

\begin{Lemma}\label{lem3zbiory}
Let $V=\overline{conv}f^s(\mathfrak{S})$, $\varepsilon>0$ and $ V^{\epsilon}=A\cup B\cup C\cup Z$, where the sets $A,B,C$ are pairwise disjoint.
Suppose that the set $B\cup C\cup Z$ is invariant and $A,C$ are the escape sets. If there exist a convex set $W\subset V^\varepsilon$ and $\delta>0$
 such that $f^s(B\cup C)\subset W$ and $W^\delta\cap(B\setminus Z)=\emptyset$ then the set $Z$ is absorbing.
\end{Lemma}
\begin{proof}
Suppose, contrary to our claim, that the set $Z$ is not absorbing.  Then
there exists a trajectory $(\overline{x}_t)_{t\geq t_0}$ such that
\[
\exists_{t_1>\max{t_0,t_{\epsilon}}} \quad \forall_{t>t_1} \quad \overline{x}_t\notin
Z
\]
where $t_{\epsilon}$ is from Corollary $\ref{lemuwypuk}$, such that $\overline{x}_t\in V^{\epsilon}$ for all $t>t_{\epsilon}$.
Since the set $B\cup C\cup Z$ is invariant, there exists
$t_2:=t_{B\cup C\cup Z}$ such that $\beta^s_t(\bar{x})\in B\cup
C\cup Z$ for every $\bar{x}\in B\cup C\cup Z$ and $t>t_2$. Since
$A$ is an escape set then
\[
\exists t_3>\max\{t_1,\,t_2\},\;\;\;\bar{x}_{t_3}\notin A
\]
So, $\bar{x}_{t_3}\in B\cup C$ and by the invariance of the set
$B\cup C\cup Z$ we obtain that
\[
\forall t\geq t_3,\;\;\;\bar{x}_t\in B\cup C\cup Z
\]
By Corollary \ref{lemuwypuk}, there exists $t_\delta>t_3$ such
that
\[
\bar{x}_t\in W^\delta \;\mbox{ for }\;t>t_\delta
\]
Since $C$ is an escape set then there exists $\bar{t}>t_\delta$
such that $\bar{x}_{\bar{t}}\notin C$. So, $\bar{x}_{\bar{t}}\in
(B\setminus Z)\cap W^\delta$, which contradicts the assumption
that $(B\setminus Z)\cap W^\delta=\emptyset$.
\end{proof}

\begin{Lemma}\label{lemzbiorB}
Assume that a set $D\subset\mathfrak{S}$ is invariant and absorbing and $D=B\cup Z$. If there exists a closed convex set $W\subset \mathfrak{S}$ and $\varepsilon>0$ such that $f^s(B)\subset W$ and $W^{\epsilon}\cap (B\backslash Z)=\emptyset$ then the set $Z$ is absorbing.
\end{Lemma}
\begin{proof}
Suppose, contrary to our claim, that the set $Z$ is not absorbing. Then
there exists a trajectory $(\overline{x}_t)_{t\geq t_0}$ such that
\[
\exists_{t_1>t_0} \quad \forall_{t>t_1} \quad \overline{x}_t\notin
Z.
\]
Since the set $D$ is invariant and absorbing then there exists
$t_2>t_1$ such that
\[
\forall t>t_2,\;\;\; \overline{x}_{t}\in D
 \]
So $\overline{x}_t\in B\setminus Z$ for $t>t_2$. Thus,
$f^s(\overline{x}_t)\in W$, for $t>t_2$ and by Corollary
\ref{lemuwypuk} there exists $t_\varepsilon>t_2$ such that
$\overline{x}_{t}\in W^{\epsilon}$ for  $t>t_\varepsilon$. This is
a contradiction to the assumption $W^{\epsilon}\cap (B\setminus
Z)=\emptyset$.

\end{proof}

\begin{Lemma}\label{lemwyjscie}
Suppose that $Z\subset \mathfrak{S}$ and $f^s(Z)\subset W$, where $W$ is a closed convex subset of $\mathfrak{S}$.
If there exists $\varepsilon>0$ such that $W^\varepsilon\cap Z=\emptyset$ then $Z$ is an escape set.
\end{Lemma}

\begin{proof}
Suppose, contrary to our claim, that $Z$ is not an escape set. So, there
exists a trajectory $(\overline{x}_{t})_{t\geq t_0}$ such that
\begin{equation}\label{zalniewprost2}
\exists \tau>t_0,\, \forall t>\tau,\;\;\;\overline{x}_{t}\in Z.
\end{equation}
So $f^s(\overline{x}_{t})\in W$ for  $t>\tau$. By Corollary
\ref{lemuwypuk}, there exists $t_{\epsilon}>\tau$ such that
$\overline{x}_{t}\in W^{\epsilon}$ for all  $t>t_{\epsilon}$,
which contradicts to (\ref{zalniewprost2}) and the assumption
$W^\varepsilon\cap Z=\emptyset$.
\end{proof}


\section{Semi-cooperative strategies}

In this section we introduce semi-cooperative strategies in the repeated PD that are a generalisation of Smale's good strategies. The semi-cooperative strategy of a player is determined by the choice of  a point  $v$ in the $\beta$-core of PD and a positive constant. If both players choose the same point $v$ then the obtained strategy profile is a Nash equilibrium and the vector payoff in the repeated game equals to $v$. This can be regarded as a very special case of Robert Aumann results presented in $\cite{aumann20}$, $\cite{aumann21}$, $\cite{aumanncore}$ and $\cite{aumannsurvey}$, where it was shown that each payoff from the $\beta$-core of the stage game can be received as a strong Nash equilibrium in the repeated game. Much more interesting is the situation when players choose semi-cooperative strategies corresponding to different points in the $\beta$-core. The limits of trajectories of the dynamical system determined by a semi-cooperative profile are described in Theorem \ref{twzbiega}, which is the main result of the paper.

To recall the definition of the $\beta$-core assume that $G$ is a normal form game like in Section \ref{seclematy}. A correlated strategy $c^{\mathcal{K}}$ of the coalition $\mathcal{K}\subset\mathcal{N}$ is a probability distributions over (the finite set) $A^{\mathcal{K}}=\prod_{i\in\mathcal{K} } A_i$. The set of correlated strategies of the coalition $\mathcal{K}$ is denoted by $C^{\mathcal{K}}$. The  correlated strategy $c^{\mathcal{K}}$ of the coalition $\mathcal{K}$ and the correlated strategy $c^{\mathcal{N}\setminus\mathcal{K}}$ of the anti-coalition $\mathcal{N}\setminus\mathcal{K}$ determine a correlated strategy $c=\left(c^{\mathcal{K}},c^{\mathcal{N}\backslash\mathcal{K}}\right)\in C^{\mathcal{N}}$ of the full coalition. In the usual way we extend payoffs functions $u_i$ onto the set of correlated strategies $C^{\mathcal{N}}$. \\
A correlated strategy $\tilde{c} \in C^{\mathcal{N}}$ belongs to the $\beta$-core $\left(\tilde{c}\in \mathcal{C}_\beta(G)\right)$ iff
\begin{equation}\label{betardzenformula}
\forall_{\mathcal{K}\subset \mathcal{N}, \mathcal{K}\neq \emptyset} \quad \exists_{c^{\mathcal{N}\backslash \mathcal{K}}\in C^{\mathcal{N}\backslash \mathcal{K}}} \quad \forall_{c^{\mathcal{K}}\in C^{\mathcal{K}}} \quad \exists_{j\in \mathcal{K}} \quad  u_j\left(c^{\mathcal{K}},c^{\mathcal{N}\backslash \mathcal{K}}\right)\leqslant u_j(\tilde{c}).
\end{equation}
Taking $\mathcal{K}=\mathcal{N}$ in (\ref{betardzenformula}) we obtain that
\[
\forall c\in C^{\mathcal{N}},\;\exists  j\in \mathcal{N},\;\;\; u_j(c)\leq u_j(\tilde{c}),
\]
which is the weak Pareto optimum condition. \\
Taking a  coalition $\mathcal{K}=\{i\}$ in (\ref{betardzenformula}) we obtain that the payoff $u(\tilde{c})$ is individually rational, i.e
\[
\forall i\in \mathcal{N},\;\;\; u_i(\tilde{c})\geq \min_{c^{\mathcal{N}\setminus\{i\}}\in C^{\mathcal{N}\setminus\{i\}}}\max_{c^{\{i\}}\in C^{\{i\}}}\,u_i(c^{\{i\}},c^{\mathcal{N}\setminus\{i\}}).
\]
Hereafter $G$ denotes the considered PD with payoffs given by (\ref{staticPD}). The set of Pareto optimal and individually rational payoffs for $G$ is illustrated in Figure \ref{beta}. The $\beta$-core for PD (in fact its image by $u$) is the sum of intervals
\[
u(\mathcal{C}_\beta(G))=\{(x,3-\frac{x}{2}):\,1\leq x\leq2\}\cup\{(3-\frac{y}{2},y):\,1\leq y\leq 2\}
\]

Let us fix a point $v$ in the $\beta$-core different to the end points, i.e
\[
v\in u(\mathcal{C}_\beta(G))\setminus\{(1,\,2.5),\,(2.5,\,1)\}.
\]
Roughly speaking, player 1 cooperates if the average payoff is located below the line going through the points $(1,\,1)$ and $v$.
A semi-cooperative strategy of player 1, determined by the point $v$ and a positive constant $\varepsilon>0$, is a map  $s_1^{v,\varepsilon}:\mathfrak{S}\rightarrow \{C,D\}$ given by
\begin{displaymath}
s_1^{v,\varepsilon}(x)=\left\{
\begin{array}{ll}
C & \quad \mbox{if }\quad   x\in K_1(v,\varepsilon)\\
D & \quad \mbox{elsewhere in } \mathfrak{S}
\end{array}
\right.,
\end{displaymath}
where
\begin{equation}\label{zbiorkooperacji1}
K_1(v,\varepsilon):=(T_1(v))^{\varepsilon}\:\cap \: \{x\in \mathfrak{S}: x_1\geqslant 1 \quad \mbox{and } \quad x_2\leqslant v_2\},
\end{equation}
and
\[
T_1(v):=\left\{x\in \mathfrak{S}: x_2\leqslant\frac{v_2-1}{v_1-1}x_1+\frac{v_1-v_2}{v_1-1}
\right\}.
\]
A semi-cooperative strategy of player 2, determined by the point $v$ and $\varepsilon>0$,  is a map $s_2^{v,\varepsilon}:\mathfrak{S}\rightarrow \{C,D\}$ given by the formula:
\begin{displaymath}
s_2^{v,\varepsilon}(x)=\left\{
\begin{array}{ll}
C & \quad \mbox{if }\quad   x\in K_2(v,\varepsilon)\\
D & \quad \mbox{elsewhere in } \mathfrak{S}
\end{array}
\right.,
\end{displaymath}
where
\begin{equation}\label{zbiorkooperacji2}
K_2(v,\varepsilon):=(T_2(t))^{\varepsilon}\: \cap\: \{x\in \mathfrak{S}: x_1\leqslant v_1 \quad \mbox{and} \quad x_2\geqslant
1\}.
\end{equation}
and
\[
T_2(v):=\left\{x\in \mathfrak{S}: x_2\geqslant\frac{v_2-1}{v_1-1}x_1+\frac{v_1-v_2}{v_1-1}
\right\}.
\]
 Semi-cooperative strategies are illustrated in  Figure $\ref{semi}$.\\
We say that player 1 is: an \textit{egoist} if $v_1>v_2$; an \textit{altruist} if $v_1<v_2$; a \textit{balanced player} if $v_1=v_2(=2)$. The second player is:  an \textit{egoist} if $v_2>v_1$; an \textit{altruist} if $v_2<v_1$; a \textit{balanced player} if $v_1=v_2(=2)$. This is illustrated in Figure \ref{e_a}.
\begin{figure}\label{e_a}
  \centering
  \includegraphics[width=0.8\textwidth]{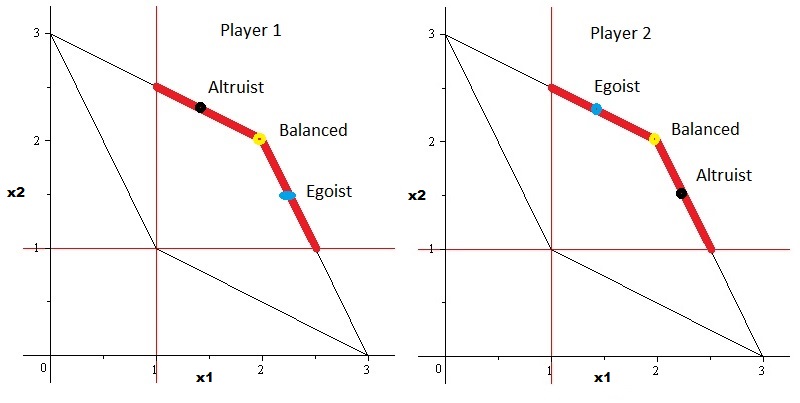}
   \caption{The models of the players' behaviour}
\end{figure}

If both players choose the same point $v=(v_1,v_2)\in u(\mathcal{C}_\beta(G))$ to determine their semi-cooperative strategies, then the strategy profile is a Nash equilibrium. We obtain  the following result which is similar to Smale's one.

\begin{Theorem}\label{twlikeSmale}
 Suppose that both players play semi-cooperative strategies $s_1^{v,\epsilon_1}, s_2^{v,\epsilon_2}$ determined by the same $v\in u(\mathcal{C}_\beta(G))\:\backslash\:$ $ \{(1,2.5),$ $(2.5,1)\}$ and positive constants $\varepsilon_1,\,\varepsilon_2>0$. If $(\bar{x}_t)_{t\geq 1}$ is an arbitrary trajectory of the dynamical system determined by the strategy profile $(s_1^{v,\epsilon_1}, s_2^{v,\epsilon_2})$, then
\[\displaystyle
\overline{x}_t \xrightarrow[t\rightarrow \infty]{} v.
\]

 If player 1 plays the semi-cooperative strategy $s_1^{v,\epsilon_1}$ and player 2 plays an arbitrary memory strategy $s_2$ then an arbitrary trajectory $(\bar{x}_t)_{t\geq 1}$ of the dynamical system determined by the strategy profile $(s_1^{v,\epsilon_1}, s_2)$ satisfies:
\[\displaystyle
\liminf_{t\rightarrow \infty } \overline{x}^1_t \geqslant 1
\qquad \mbox{ and }\qquad
\limsup_{t\rightarrow \infty } \overline{x}^2_t \leqslant v_2.
\]
\end{Theorem}

\begin{proof}
Fix $v\in u(\mathcal{C}_\beta(G))\:\backslash\:$ $ \{(1,2.5),$ $(2.5,1)\}$ and $\varepsilon_1,\,\varepsilon_2>0$.
Let $(\bar{x}_t)_{t\geq 1}$ be a trajectory of the dynamical system $\beta^{s^v}$, where $s^{v}\:=\:(s_1^{v,\epsilon_1}, s_2^{v,\epsilon_2}) $ and $f^v=u\circ s^v$.

 Let
\begin{displaymath}
\begin{array}{lcl}
\Delta & := & K_1(v,\epsilon_1)\cap K_2(v,\epsilon_2)
\end{array}
\end{displaymath}
\begin{displaymath}
\begin{array}{lcl}
\Omega_1 & := & \mathfrak{S} \backslash K_1(t_1,\epsilon_1)\\
\Omega_2 & := & \mathfrak{S} \backslash K_2(t_2,\epsilon_2)
\end{array}
\end{displaymath}
The set $\Delta$ is convex; the sets $\Delta$, $\Omega_1$, $\Omega_1$ are pairwise disjoint and $\mathfrak{S}\:=\:\Delta\:\cup\:\Omega_1\:\cup\:\Omega_1$. This situation is illustrated in Figure $\ref{tosamot}$.
\begin{figure}[!ht]
  \centering
    \includegraphics[width=0.6\textwidth]{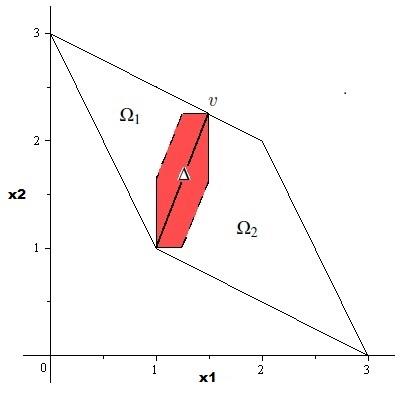}
  \caption{The partition of $\mathfrak{S}$}
\label{tosamot}
\end{figure}

Observe that
\begin{equation}\label{dynamika1}
f^{v}\left(\overline{x}_t\right)=\left\{
\begin{array}{lcl}
(2,2) & \mbox{ if } & \overline{x}_t\in \Delta \\
(3,0) & \mbox{ if } & \overline{x}_t\in \Omega_1 \\
(0,3) & \mbox{ if } & \overline{x}_t\in \Omega_2 \\
\end{array}\right..
\end{equation}
To show that the trajectory  $(\bar{x}_t)_{t\geq 1}$ converges to
$v$ we construct an invariant and absorbing neighborhood
$O_{\delta}(v)$ of the point $v$. Fix
$\delta\in\left(0,\min\left\{\epsilon_1, \epsilon_2
\right\}\right)$. We will denote by $l(a,\,b)$ the line going
through the points $a,\,b\in \mathbb{R}^2$. Set
\begin{equation}\label{dowodoznaczenia}
\begin{array}{lcl}
l_1  &=&l(P_{\delta},\,(3,0))\\
l_2  &=&l(P_{\delta},\,(0,3))\\
l_3  &=&l(R_{\delta},\,(2,2))\\
\end{array}
\end{equation}
where $P_{\delta}  = (v_1-\delta,v_2-\delta)$ and $R_{\delta}$ is
the intersection point of $l_2$ and the line $\{x_1=v_1\}$. The
half-plane over (under) a line $l\subset \mathbb{R}^2$ will be
denoted by $e(l)$ ($h(l)$). The neighborhood $O_{\delta}(v)$
illustrated in Figure \ref{rysotoczenie} is given by
\begin{equation}\label{otoczeniey}
O_{\delta}(v):= {e(l_1)} \: \cap \: {e(l_2)} \: \cap \: {e(l_3)}\:\cap\: \{x\in\mathfrak{S}:\;  x_1\leqslant v_1+\delta\}.
\end{equation}
\begin{figure}[!ht]
  \centering
    \includegraphics[width=0.6\textwidth]{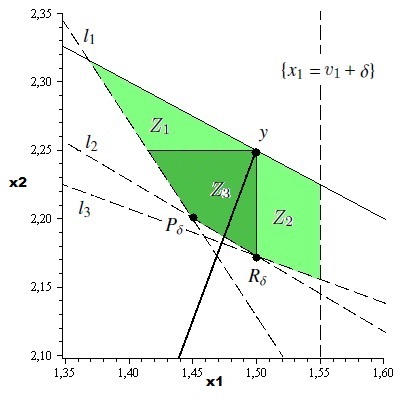}
  \caption{The neighbourhood $O_{\delta}(v)$ }
\label{rysotoczenie}
\end{figure}

To show that the neighbourhood $O_{\delta}(v)$ is invariant we
divide it into three parts: $Z_1:=\Omega_1\cap O_{\delta}(v)$,
$Z_2:=\Omega_2\cap O_{\delta}(v)$,   $Z_3:= \Delta\cap
O_{\delta}(v)$.  By  (\ref{malykrok}) there exists $t_1>1$ such
that for any $t>t_1$
\begin{equation}\label{krok}
|\beta^{s^v}_t(\overline{x})-\overline{x}|<\delta.
\end{equation}
If $\bar{x}\in Z_1$ then $f^v(\bar{x})=(3,0)$. So $\beta^{s^v}_t(\overline{x})\in e(l_1)$. If $t>t_1$ then $\beta^{s^v}_t(\overline{x})\in\{x\in \mathfrak{S}:\,x_1<v_1+\delta,\,x_2>v_2-\delta\}$. Since $e(l_1)\cap \{x\in \mathfrak{S}:\,x_1<v_1+\delta,\,x_2>v_2-\delta\}\subset O_\delta(v)$ we obtain that $\beta^{s^v}_t(\overline{x})\in O_\delta(v)$.\\

If $\bar{x}\in Z_2$ then $f^v(\bar{x})=(0,3)$. So $\beta^{s^v}_t(\overline{x})\in e(l_2)\cap e(l_3)$. If $t>t_1$ then $\beta^{s^v}_t(\overline{x})\in\{x\in \mathfrak{S}:\,v_1-\delta<x_1<v_1+\delta\}$.
Since $e(l_2)\cap e(l_3)\cap \{x\in \mathfrak{S}:\,v_1-\delta<x_1<v_1+\delta\}\subset O_\delta(v)$ we obtain that $\beta^{s^v}_t(\overline{x})\in O_\delta(v)$.\\

If $\bar{x}\in Z_3$ then $f^v(\bar{x})=(2,2)$. So $\beta^{s^v}_t(\overline{x})\in e(l_1)\cap e(l_2)\cap e(l_3)$. If $t>t_1$ then $\beta^{s^v}_t(\overline{x})\in\{x\in \mathfrak{S}:\,x_1<v_1+\delta\}$.
Since $e(l_1)\cap e(l_2)\cap e(l_3)\cap \{x\in \mathfrak{S}:\,x_1<v_1+\delta\}= O_\delta(v)$ we obtain that $\beta^{s^v}_t(\overline{x})\in O_\delta(v)$.\\
So the neighbourhood $O_{\delta}(v)$ is invariant.

 To show that the set $O_{\delta}(v)$ is absorbing we set $V=conv\:f^v(\mathfrak{S})=conv\{(2,2),(0,3),(3,0)\}$ and
\begin{equation}\label{notation1}
A=\Omega_1\cap V^{\delta},\quad  B=\Delta\cap V^{\delta},\quad C=\Omega_2\cap V^{\delta}, \quad Z=O_{\delta}(y), \quad W=conv\{(2,2),(0,3)\}.
\end{equation}
The set  $B\cup C\cup Z$ is invariant  and  $A$ and $C$ are  escape sets. We have  $f^v(B\cup C)\subset W$. There exists a constant $\theta>0$ such that
$W^{\theta} \cap (B\backslash Z)= \emptyset$ (See Figure $\ref{rysoddzielne}$).
\begin{figure}[!ht]
  \centering
    \includegraphics[width=0.6\textwidth]{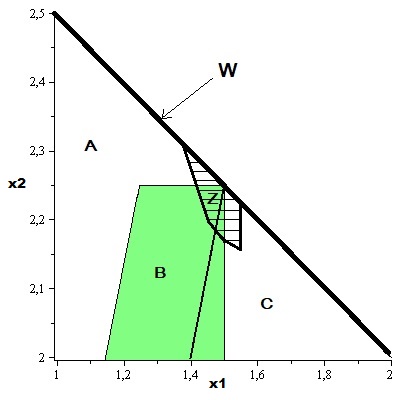}
  \caption{The illustration of the condition $W^{\theta} \cap (B\backslash Z)= \emptyset$}
\label{rysoddzielne}
\end{figure}
By Lemma \ref{lem3zbiory}, we obtain that the set
$ Z\:=\:O_{\delta}(v)$ is absorbing.

Since the diameters of the neighbourhoods $O_{\delta}(v)$ tends to
zero as $\delta\to 0$, then we obtain the convergence of the
trajectory $(\bar{x}_t)$ to the point $v$.

Now, we consider the dynamics of the system when player 2 chooses an arbitrary memory strategy and player 1 plays the semi-cooperative strategy
$s_1^{v,\epsilon_1}$. If $\bar{x}^2_t>v_2$ then player 1 defects in the next stage and so, player's 2 payoff belongs to $\{0,\,1\}$. By Corollary \ref{lBC}, we obtain that
\[
\limsup_{t\rightarrow \infty } \overline{x}^2_t \leqslant v_2.
\]
 The proof of the inequality
 \[
 \liminf_{t\rightarrow \infty } \overline{x}^1_t \geqslant 1
 \]
  is similar.
\end{proof}

If players have no opportunity to agree the choice of the point $v$ then we should not expect that they will choose the same point. We can treat the choice of the point $v$ as an action of a player in a new game. To define payoffs in this new game we have to know the payoffs in the repeated PD when players 1 and 2 play semi-cooperative strategies corresponding to points $a$, $b$, respectively,  in the $\beta$-core. The main result of the paper concerns this situation in the following.

\begin{Theorem}\label{twzbiega}
Let $a=(a_1,a_2),b=(b_1,b_2)\in u(\mathcal{C}_\beta(G))\:\backslash\:$ $ \{(1,2.5),$ $(2.5,1)\}$  and $\varepsilon_1,\varepsilon_2>0$. If $(\bar{x}_t)$ is an arbitrary trajectory of the dynamical system determined by the strategy profile  $(s_1^{a,\varepsilon_1},\,s_2^{b,\varepsilon_2})$, then
\begin{displaymath}
\overline{x}_t\quad \xrightarrow[t\rightarrow \infty]{} \quad \left\{
\begin{array}{lcl}
a & \mbox{ if } & a_1=b_1  \\
(2,2) & \mbox{ if } & a_1 \leqslant 2 \leqslant b_1  \\
b&\mbox{ if } & a_1<b_1\leq 2 \\
a &\mbox{ if } & 2 \leq  a_1<b_1 \\
y^{\varepsilon_1,\varepsilon_2} &\mbox{ if } & b_1<a_1
\end{array}\right.
\end{displaymath}
where $y^{\varepsilon_1,\varepsilon_2}$ is a point in $\mathfrak{S}$ and
\[
y^{\varepsilon_1,\varepsilon_2} \xrightarrow[\varepsilon_1,\varepsilon_2\rightarrow0]{} (1,1).
\]
\end{Theorem}
The case $a=b$ was considered in Theorem $\ref{twlikeSmale}$.

\begin{proof}
Set $s^{*}\:=\:(s_1^{a,\epsilon_1}, s_2^{b,\epsilon_2}) $, $f^*=u\circ s^*$ and
\begin{equation}\label{dobordelta}
\delta_0:=\left\{
\begin{array}{lcl}
\min\left\{\frac{\epsilon_1}{4},\frac{\epsilon_2}{4}, \frac{||a-b||}{4}\right\} & \mbox{if} & a_1 < b_1\leqslant 2 \mbox{ or } 2\leqslant a_1 < b_1 \\
\min\left\{\frac{\epsilon_1}{4},\frac{\epsilon_2}{4}, \frac{||a-(2,2)||}{4},\frac{||b-(2,2)||}{4}\right\} & \mbox{if} & a_1 \leqslant 2 \leqslant b_1  \\
\min\left\{\frac{\epsilon_1}{4},\frac{\epsilon_2}{4}, \frac{||y^{\varepsilon_1,\varepsilon_2}-(1,1)||}{4}\right\} & \mbox{if} & b_1> a_1
\end{array}\right.
\end{equation}
For $\delta\in (0,\delta_0)$ there exists $t_\delta>1$ such that the condition $(\ref{krok})$ is satisfied for  $t>t_\delta$ .\\
We denote by $\beta^*$ the dynamical system determined by the strategy $s^*$, i.e. $\beta^*_t(\bar{x})=\frac{t\bar{x}+f^*(\bar{x})}{t+1}$.
Set
\begin{equation}\label{zbiordelta}
\begin{array}{lcl}
\Delta & := & K_1(a,\epsilon_1)\cap K_2(b,\epsilon_2)\\
\Omega_1 & := & K_2(b,\epsilon_2) \backslash \Delta\\
\Omega_2 & := & K_1(a,\epsilon_1) \backslash \Delta \\
\Omega_3 & := & \mathfrak{S} \backslash (K_1(a,\epsilon_1) \cup K_2(b,\epsilon_2))
\end{array}.
\end{equation}
If $a_1\leq b_1$ then $\Omega_3=\emptyset$. The case $a_1=b_1$ has been proved in Theorem \ref{twlikeSmale}. \\
In the case $a_1< 2\leq b_1$ we
set $V=conv\{(0,\,3),\,(2,\,2),\,(3,\,0)\}$ and observe that $V^\varepsilon\cap \Delta$ is absorbing (comp. Corollary \ref{lemuwypuk}). Since $f^*(x)=(2,\,2)$ for $x\in V^\varepsilon\cap \Delta$ and $(2,\,2)\in V^\varepsilon\cap \Delta$ then every trajectory is convergent to $(2,\,2)$. \\
The case $a_1\leq 2<b_1$ is symmetric to the above one.

The case $a_1<b_1<2$ is illustrated in Figure $\ref{rysegoaltr}$.
\begin{figure}[!ht]
  \centering
    \includegraphics[width=0.6\textwidth]{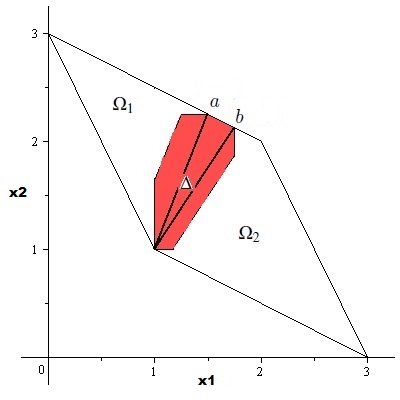}
  \caption{The partition of $\mathfrak{S}$ when $a_1<b_1<2$}
\label{rysegoaltr}
\end{figure}

We  construct two neighbourhoods of the point $b$:
\[
O_{\delta}(b):={e(l_1})\:\cap\:{e(l_2)}\:\cap\: {e(l_3)}\:\cap\:
\{x\in \mathfrak{S};\,  x_1\leqslant b_1+\delta\}.
\]
\[
U_{\delta}(b):= O_\delta(b)\:\cap \:\{x\in \mathfrak{S}:\,
b_1-\delta \leqslant x_1\}.
\]
where the lines $l_1,\,l_2,\,l_3$ are given by
(\ref{dowodoznaczenia}) for  $P_{\delta}  =
(a_1-\delta,a_2-\delta)$ and $R_{\delta}$ being the intersection
point of $l_2$ and the line $\{x_1=b_1\}$. Invariance of the sets
$U_\delta(b)$, $O_\delta(b)$ is a conclusion from its construction
(see Figure \ref{otoczU}).
\begin{figure}[!ht]
  \centering
    \includegraphics[width=0.6\textwidth]{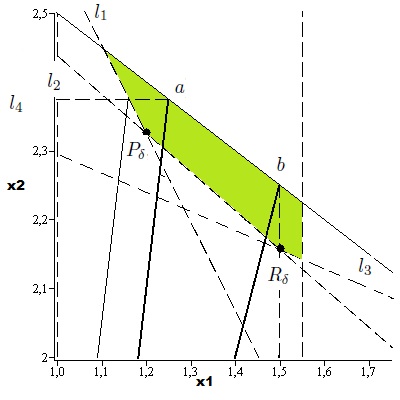}
  \caption{The neighbourhoods $O_\delta(b)$}
\label{otoczU}
\end{figure}
By Lemma \ref{lem3zbiory}, we obtain that the set $O_\delta(b)$ is
absorbing using the same notation as in (\ref{notation1}). To
obtain that $U_\delta(b)$ is absorbing we apply Lemma \ref{lemzbiorB}
setting $D=O_\delta(b)$, $Z=U_\delta(b)$, $B=D\setminus Z$,
$W=conv\{(3,\,0),\,(2,\,2)\}$.

The case $2< a_1 < b_1$ is symmetric to the case
 considered above.

 The last situation is $b_1<a_1$. Two possible choices of $a$ and $b$ are presented in Figure \ref{zb11}.
\begin{figure}[!ht]
  \centering
    \includegraphics[width=0.8\textwidth]{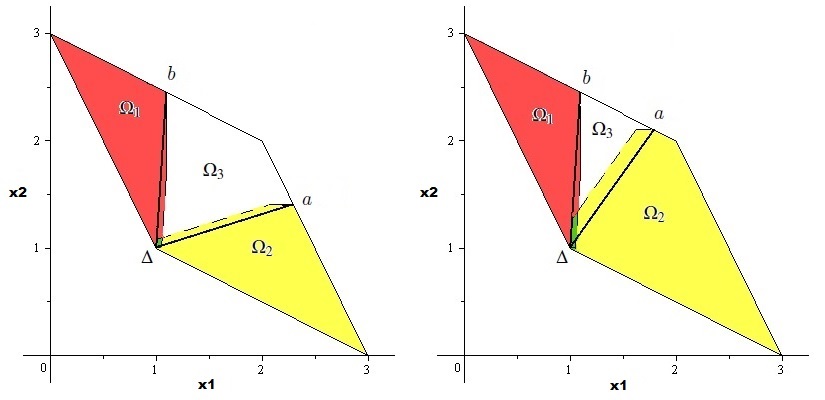}
  \caption{On the left -- the partition of $\mathfrak{S}$ when both players are egoists and on the right -- the partition of $\mathfrak{S}$ when $b_1<a_1$ and player 1 is the altruist and the player 2 is the egoist }
\label{zb11}
\end{figure}
Obviously, we have
\begin{displaymath}
f^*(\overline{x})=\left\{
\begin{array}{lcl}
(2,2) & \mbox{ if } & \overline{x}\in \Delta \\
(3,0) & \mbox{ if } &  \overline{x}\in \Omega_1 \\
(0,3) & \mbox{ if } &  \overline{x}\in \Omega_2 \\
(1,1) & \mbox{ if } &  \overline{x}\in \Omega_3 \\
\end{array}\right..
\end{displaymath}
The limit of an arbitrary trajectory is a point
$y=y^{\varepsilon_1\,\varepsilon_2}$ which is the intersection
$cl(\Delta) \:\cap\: cl(\Omega_3)$. Let $P_{\delta}\in \Omega_3$
be the unique point satisfying $dist(P_{\delta},\Omega_1)=
dist(P_{\delta},\Omega_2)=\delta$. Set
\begin{equation}\label{dowodoznaczenia2}
\begin{array}{l}
l_1 =l(P_{\delta},\,(0,\,3))\\
l_2 =l(P_{\delta},\,(3,\,0))\\
l_3 =l((1,\,1),b)\\
l_4=l((1,\,1),a)\\
l_5=l(P_\delta,y^{\varepsilon_1\,\varepsilon_2})
\end{array}
\end{equation}
 We construct the neighbourhood of the
set $\Delta$:
\begin{displaymath}
O_{\delta}(\Delta):= \Delta^{\delta} \:\cup \:(h(l_1)\:\cap\:
h(l_2)\:\cap\: h(l_3)\:\cap\: e(l_4))
\end{displaymath}
which is illustrated in Figure $\ref{otoczeniedelty}$.

\begin{figure}[!ht]
  \centering
    \includegraphics[width=0.6\textwidth]{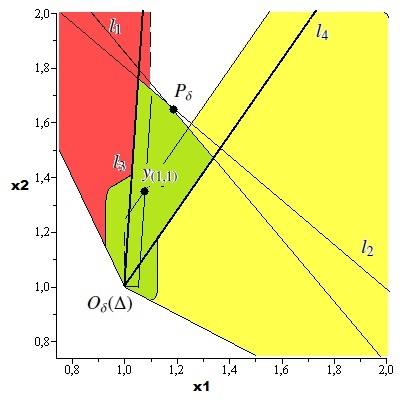}
  \caption{The neighborhood $O_{\delta}(\Delta)$}
\label{otoczeniedelty}
\end{figure}

We show that $O_{\delta}(\Delta)$ is invariant. Fix $\bar{x}=(x_1,\,x_2)\in O_\delta(\Delta)$ and $t>t_\delta$. Set $C:=h(l_1)\cap h(l_2)\cap h(l_3)\cap e (l_4)$.\\
Consider the case $\bar{x}\in\Omega_1$ and $x_2>y_2$. Then $\bar{x}$ and $f^*(\bar{x})=(3,\,0)$ belong to $h(l_2)\cap h(l_3)$. So $\beta^*_t(\bar{x}) \in h(l_2)\cap h(l_3)$. Since $dist(\bar{x},h(l_4))\geq\varepsilon_1>\delta$ then $dist(\beta^*_t(\bar{x}),\,h(l_4))>0$ and so $\beta^*_t(\bar{x})\in e(l_4)$. Since $dist(\bar{x},\,e(l_3))<\varepsilon_2$ then $dist(\beta^*_t(\bar{x}),\,e(l_3))<\varepsilon_2+\delta$. We have $\{z\in \mathfrak{S}:\;dist(z,e(l_3))<\varepsilon_2+\delta\}\cap h(l_2)\subset h(l_1)$. So $\beta^*_t(\bar{x})\in h(l_1)$. Thus we obtain that $\beta^*_t(\bar{x})\in C\subset O_\delta(\Delta)$.\\
If $\bar{x}\in\Omega_1$ and $x_2\leqslant y_2$ then $\beta^*_t(\bar{x})\in O_\delta(\Delta)$.\\
The case $\bar{x}\in\Omega_2$ is symmetric to the above one.\\
In the case $\bar{x}\in\Omega_3$ observe that $\bar{x}$ and $f^*(\bar{x})=(1,\,1)$ belong to the convex set $C$. Thus $\beta^*_t(\bar{x})\in C$.\\
In the last case $\bar{x}\in\Delta$, we have $\beta^*_t(\bar{x})\in\Delta^\delta$.

To obtain that $O_{\delta}(\Delta)$ is absorbing we set

\begin{equation}\label{dowodoznaczenia4}
\begin{array}{lcll}
A & := & e(l_5) \:\backslash\: O_{\delta}(\Delta) \:\subset \:\Omega_1\cup\Omega_3, \\
B & := & (h(l_5)\:\cap\: \Omega_3)\:\backslash\: O_{\delta}(\Delta), \\
C & := & \Omega_2 \:\backslash \:O_{\delta}(\Delta),\\
Z & := & O_{\delta}(\Delta),\\
V & := &\mathfrak{S},\\
W & := & conv\{(1,1),(0,3)\}
\end{array}
\end{equation}
and we apply Lemma \ref{lem3zbiory}.
The above sets are illustrated in the Figure $\ref{podzial}$.
\begin{figure}[!ht]
  \centering
    \includegraphics[width=0.6\textwidth]{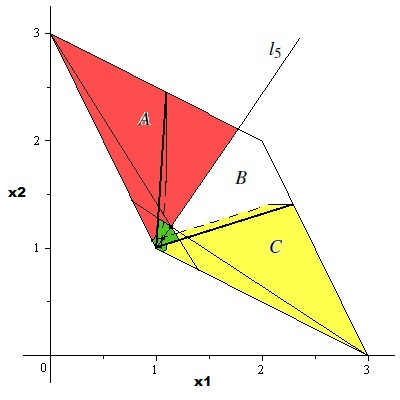}
  \caption{The sets $A$, $B$ and $C$ in one possible situation when $b_1<a_1$}
\label{podzial}
\end{figure}
The set $B\cup C\cup Z$ is invariant  and the sets $A$, $C$ are escape sets.
We have $f^*(B\cup C)\subset
W$. The neighbourhood $O_{\delta}(\Delta)$ is defined in such way
that there exists a constant  $\theta>0$ such that
\[
W^{\theta} \cap (B\backslash Z)=W^{\theta} \cap B = \emptyset.
\]
By Lemma \ref{lem3zbiory}, the neighbourhood $O_{\delta}(\Delta)$ is absorbing.

We define the neighbourhood $O_{\delta}(y)$ of the point $y$ by
\[
O_{\delta}(y) := O_{\delta}(\Delta) \cap e(l_6),
\]
where $l_6$ is the line given by the equation: $x_1+x_2=y_1+y_2-\delta$.\\
Using similar arguments as in the proof of the invariance of $O_\delta(\Delta)$ we  show that $O_{\delta}(y)$ is invariant.
Set $D=O_{\delta}(\Delta)$, $Z=O_{\delta}(y)$, $B=D\setminus Z$ and $W=conv\{(0,\,3),\,(2,\,2),\,(3,\,0)\}$. By Lemma \ref{lemzbiorB}, we obtain that $O_{\delta}(y)$ is absorbing.\\
Since the diameter of  $O_\delta(y)$ tends to zero as $\delta\to 0$, then $y$ is the limit of an arbitrary trajectory in the considered case.

Using elementary calculations we obtain that the distance between the point $y^{\varepsilon_1\,\varepsilon_2}$ and the point $(1,\,1)$ equals to
\[
\frac{\varepsilon_1+\varepsilon_2+\sqrt{(\varepsilon_1+\varepsilon_2)^2+4\varepsilon_1\varepsilon_2\tan^2\alpha}}{2\tan\alpha},
\]
where $\alpha$ denotes the angle between lines $l((1,\,1),\,a)$ and $l((1,\,1),\,b)$.
\end{proof}



\end{document}